\documentclass{article}
\usepackage[T1]{fontenc}
\usepackage[utf8]{inputenc}
\usepackage{amsfonts,amsmath,amsthm,hyperref}
\usepackage[none]{hyphenat}
\hypersetup{
    colorlinks,
    linkcolor={blue},
    citecolor={blue},
    urlcolor={blue}
}
\usepackage[noblocks]{authblk}
\newtheorem{thm}{Theorem}
\newtheorem{note}{Note}
\newtheorem{rem}{Remark}
\newtheorem{cor}{Corollary}
\newtheorem{prop}{Proposition}

\newcommand{\R}{\mathbb{R}}
\newcommand{\Lo}{\mathbb{L}}
\newcommand{\B}{\mathcal{B}}
\newcommand{\E}{\mathcal{E}}
\newcommand{\F}{\mathcal{F}}
\newcommand{\G}{\mathcal{G}}
\newcommand{\eL}{\mathcal{L}}
\newcommand{\M}{\mathcal{M}}
\newcommand{\N}{\mathcal{N}}
\newcommand\norm[1]{\left\Vert#1\right\Vert}
\newcommand{\sign}{\text{sgn}}

\title{Born-Infeld Solitons and Existence \& Non-uniqueness of Solutions to the Bj\"{o}rling Problem}
\author{Arka Das\thanks{Indian Institute of Science, Bengaluru, Karnataka, India}}

\date{}

\begin{document}

\maketitle
\begin{abstract}
In this semi-expository article, we study Born-Infeld soliton surfaces as zero mean curvature surfaces and derive conformal parameters for them. Then we present two approaches to solve the Bj\"{o}rling problem for such surfaces, one of them treating them as time-like minimal surfaces and the other one using the Barbashov-Chernikov representation. Finally, we show that the solution to the Bj\"{o}rling problem may not be unique unlike minimal and maximal surfaces.
\end{abstract}

\section{Introduction}
We know that zero mean curvature surfaces in Euclidean space $\mathbb{R}^3$ are minimal surfaces. Similarly, zero mean curvature surfaces in Lorentz-Minkowski space $\mathbb{L}^3$ ($ds^2 = dx^2 + dy^2 - dz^2$) are maximal (which are spacelike) and timelike minimal surfaces. In this article we focus on the Born-Infeld solitons. First consider the Born-Infeld equation:
\begin{equation}\label{BI-graph}
\left(1-\phi_y^2\right)\phi_{xx}+2\phi_x\phi_y\phi_{xy}-\left(1+\phi_x^2\right)\phi_{yy}=0.
\end{equation}
From \cite{whit}, we know that this PDE is hyperbolic if $(\phi_x^2-\phi_y^2+1)>0$. In this article, we will say $\phi=\phi(x,y)$ defined over some open set $\Omega\subseteq \mathbb{R}^2$ is a BI soliton if it satisfies \eqref{BI-graph} and if $(\phi_x^2-\phi_y^2+1)>0$.

By BI soliton surface, we will denote a surface $X:\Omega\to \mathbb{R}^3$ with $X(x,y)=(x,y,\phi(x,y)$, where $\Omega$ is an open set in $\R^3$ and $\phi$ is a BI soliton. Later we will give a definition for generalized Born-Infeld solition surfaces which may not always be a graph of a function. Born-Infeld solitons occur naturally in physics. 

In ~\cite{dey}, Dey and Singh  showed the following proposition for a surface in $\Lo ^3$.
\begin{prop}[Dey \& Singh]\label{DS}
The  Born-Infeld solitons can be represented as a spacelike minimal graph or timelike minimal graph over a domain in timelike plane or a combination of both away from singular points (points where the tangent plane degenerates), i.e., points where the determinant of the coefficients of the first fundamental form vanishes.
\end{prop}
If we impose the condition $(\phi_x^2-\phi_y^2+1)>0$, a BI soliton $\phi$ can be represented only as a timelike minimal graph  and not as a spacelike graph. (We will prove this in the next section.)

We endow $\mathbb{R}^3$ with an inner product structure $\langle \cdot , \cdot \rangle_{\B^3}$, defined as $\langle v, w\rangle_{\B^3}=v_1 w_1 -v_2 w_2 +v_3 w_3$. We call this space as $\B^3$. We get a natural cross-product $\times_{\B^3}$ in $\B^3$, that satisfies 
$$
\langle u\times_{\B^3} v, w\rangle_{\B^3}=\det\left(\big[u\; v\; w\big]\right).
$$

We will first show that BI soliton surfaces in $\B^3$ have mean curvature zero.

Next we define the generalized Born-Infeld soliton surfaces using the general parametrization given by Barbashov and Chernikov. 

We may call this parametrization as Barbashov-Chernikov parametrization.

We show that if the generalized Born-Infeld soliton surface is a graph, then it satisfies equation \eqref{BI-graph} and gives a timelike minimal surface away from a subset of domain. We derive conformal parameters for the generalized Born-Infeld soliton surfaces. 

Next we study the Bj\"{o}rling problem for Born-Infeld soliton and generalized BI soliton surfaces using two approaches. 

In the first approach, we use the fact that $(x,y,\phi(x,y))$ will give a BI soliton surface in $\B^3$ if and only if $(\phi(x,y),x,y)$ gives a timelike minimal surface in $\Lo^3$.  We use the Bj\"{o}rling problem and its solution for the timelike minimal surface ~\cite{kim} to get the solution for the BI soliton surface under some sufficient conditions. In \cite{dev}, Manikoth used this approach (using a different method) to get certain sufficient conditions for the  existence of a solution to Bj\"{o}rling problem for BI soliton surfaces. In the second approach we use the Barbashov-Chernikov representation to find some sufficient condition for existence of a solution to Bj\"{o}rling problem for generalized BI soliton surfaces.\\
Finally, we show that the solution to the Bj\"{o}rling problem for generalized BI soliton surfaces is not unique. We show that this is not inconsistent with the result proved in \cite{cha} (where the uniqueness of the solution for a timelike minimal surface was proved in a rhombus shaped domain).

\section{Born-Infeld Soliton Surfaces}

\subsection{Some preliminaries of Born-Infeld soliton surfaces}

\begin{thm}\label{ZMC}
BI soliton surfaces are zero mean curvature surfaces in $\B^3$. 
\end{thm}
\begin{proof}
For a BI soliton surface $X(x,y)=(x,y,\phi(x,y))$, by definition $\phi_x^2-\phi_y^2+1>0 $. Note that $X_x=(1,0,\phi_x)$ and $X_y=(0,1,\phi_y)$. From these, we can calculate the coefficients of the first fundamental form: $\E:=\langle X_x, X_x\rangle_{\B^3} =1+\phi_x^2$, $\F:=\langle X_x, X_y \rangle_{\B^3}  = \phi_x \phi_y$ and $\G:=\langle X_y, X_y \rangle_{\B^3} = -1+\phi_y^2$. The unit normal $N:=\frac{X_x\times X_y}{\norm{X_x\times X_y}}=\frac{(-\phi_x, \phi_y, 1)}{\sqrt{(\phi_x^2-\phi_y^2+1)}}$ satisfies $\langle N, N\rangle _{\B^3}=1$. Write $K=\sqrt{(\phi_x^2-\phi_y^2+1)}>0$. The coefficients of the second fundamental form are calculated to be  $\eL=\langle X_{xx},N\rangle_{\B^3}=\frac{\phi_{xx}}{K}$, $\M=\langle X_{xy},N\rangle_{\B^3}=\frac{\phi_{xy}}{K}$ and $\N=\langle X_{yy},N\rangle_{\B^3}=\frac{\phi_{yy}}{K}$. The mean curvature of $X$ is $H=\left(\frac{\eL\G-2\M\F+\N\E}{2(\E\G-\F^2)}\right)$. As $\E\G-\F^2$ $=-K^2<0$ and $\eL\G-2\M\F+\N\E=\frac{-\left(\left(1-\phi_y^2\right)\phi_{xx}+2\phi_x\phi_y\phi_{xy}-\left(1+\phi_x^2\right)\phi_{yy}\right)}{K}=0$, from equation (\ref{BI-graph}), we get that $X$ is a zero mean curvature surface.
\end{proof}

In \cite{whit}, Whitham describes a method given by Barbashov and Chernikov for solving the Born-Infeld equation (\ref{BI-graph}). The method involves interchanging the dependent and independent variables. To be able to do this at least locally, we need the condition (from Inverse Function Theorem) that $\phi_{xx}\phi_{yy}-\phi_{xy}^2\neq 0$. The functions $\phi(x,y)=f(x+y)$ or $\phi(x,y)=g(x-y)$ for arbitrary $C^2$ functions $f$ or $g$ do not satisfy this condition and although they satisfy the Born-Infeld equation, they are not included in the general solution obtained by this method. 

In this method, initially two new variables $u:=\frac{\phi_x-\phi_y}{2}$ and $v:=\frac{\phi_x+\phi_y}{2}$ and using these, two more new variables $r:=\frac{\sqrt{1+4uv}-1}{2v}$ and $s:=\frac{\sqrt{1+4uv}-1}{2u}$ are defined.

The general solution obtained by this method is given by

\begin{align}
&x-y=F(r)-\int s^2 G'(s) ds\label{x}
\\
&x+y=G(s)-\int r^2 F'(r) dr \label{y}
\\
&\phi=\int rF'(r) dr + \int sG'(s) ds \label{p}
\end{align}
where $F$ and $G$ are arbitrary functions.
\begin{note}\label{BI graph}
For arbitrary $C^2$ functions $F$ and $G$, we may not get a BI soliton, as we may not be able to write $\phi$ as a function of $x$ and $y$. However, if the map $(r,s)\mapsto (x,y)$ is invertible (at least locally), we will get a BI soliton. More simply, if the map $(r,s)\mapsto (x-y,x+y)$ is locally invertible, we will get a BI soliton. A sufficient condition is given by IFT: If for some $r_0$ and $s_0$, $(r_0^2 s_0^2-1) F'(r_0) G'(s_0)\neq 0$, the above formulae gives a BI soliton on a neighborhood of $(r_0,s_0)$.

\end{note}

Observe that $r^2s^2=1$ implies that either $rs=1$ or $rs=-1$. Now, if we substitute the expression of $r$ and $s$ given in \cite{whit} in any of the two equations and simplify, we see that either equation is equivalent to $4uv=\phi_x^2-\phi_y^2=0$, where $u=\frac{\phi_x-\phi_y}{2}$ and $v=\frac{\phi_x+\phi_y}{2}$. However, since $rs=\frac{\left(\sqrt{1+4uv}-1\right)^2}{4uv}$, either $rs$ will be undefined or in case the limit of $rs$ exists as $uv\to 0$, the limiting value of $rs$ will be 0, which is a contradiction! So, the points $(r,s)$ with $r^2s^2=1$ are forbidden in the sense that no point on a BI soliton surface can correspond to a point $(r,s)$ with $r^2s^2=1$. We will discuss more about this in the next subsection.

\begin{rem}\label{BI and TMS}
If we rename $x$, $y$ and $z$ axes as $y$, $z$ and $x$ axes respectively, we will notice that the space $\B^3$ becomes Lorentz-Minkowski space $\Lo^3$ with Lorentzian metric. Moreover, after this renaming, any BI soliton surface $X$ given by $X(x,y)=(x,y,\phi(x,y))$ (in $\B^3$) can be written as a surface $Y$ given by $Y(y,z)=(\phi(y,z),y,z)$ (in $\Lo^3$).  $Y$ is a zero mean curvature surface in $\Lo^3$ (by a calculation similar to those in the proof of Theorem \ref{ZMC} or Proposition \ref{DS}) and as the unit normal $N_Y$ to $Y$ satisfies $\langle N_Y,N_Y\rangle_{\Lo^3}=\langle N_X,N_X\rangle_{\B^3} =1$ (from proof of Theorem \ref{ZMC}), $Y$ is timelike. Hence, it is a timelike minimal surface. The converse also holds true i.e. f we have a timelike minimal surface $Y$ of the form $(\phi(y,z),y,z)$, then $\phi$ is a BI soliton. This result was proved by Dey and Singh in \cite{dey}. Similarly, if we rename $x$, $y$ and $z$ axes as $x$, $z$ and $y$ axes respectively, then the space $\B^3$ becomes Lorentz-Minkowski space $\Lo^3$ with Lorentzian metric. So, if $\phi$ is a Born-Infeld soliton, $(x,\phi(x,z),z)$ will define a timelike minimal surface and conversely, for any timelike minimal surface of the form $(x,\phi(x,z),z)$, $\phi$ satisfies the BI equation. In \cite{dev}, Manikoth showed that any regular timelike minimal surface can be written locally as a graph over either $yz$ or $xz$ plane with the height function being a Born-Infeld soliton.
\end{rem}

\subsection{Domain of BI Soliton in \textit{rs} Plane}
We will see that for a Born-Infeld soliton $\phi$, $r$ and $s$ cannot take any values. Thus, we will find the domains in $rs$ plane, over which a BI soliton can be defined by $\eqref{x}$, $\eqref{y}$ and $\eqref{p}$. The first obvious things to note is that: for $u=\frac{\phi_x-\phi_y}{2}$ and $v=\frac{\phi_x+\phi_y}{2}$, $1+4uv=1+\phi_x^2-\phi_y^2>0$. Using the definitions $r:=\frac{\sqrt{1+4uv}-1}{2v}$ and $s:=\frac{\sqrt{1+4uv}-1}{2u}$, we obtain the following result.
\begin{thm}\label{rs domain}
The domain (whenever defined) of any BI soliton in the $rs$ plane is a subset of $\{(r,s): |rs|< 1\}$.
\end{thm}
\begin{proof}
From the definitions of $r$ and $s$, we get that $rs+1=\frac{\left(\sqrt{1+4uv}-1\right)^2}{4uv}+1=\frac{2(t^2-t)}{t^2-1}$, where $t=\sqrt{1+4uv}>0$. So, for $t\neq 1$, $rs+1=\frac{2t}{t+1}>0$ and this holds for $t\to 1$ as well. (For $t=1$, $rs+1$ can be defined only if its limit exists as $t\to 1$.) So, $rs>-1$.\\
Similarly, $rs-1=\frac{\left(\sqrt{1+4uv}-1\right)^2}{4uv}-1=\frac{2(1-t)}{t^2-1}=-\frac{2}{t+1}<0$ for $t\neq 0$. Hence, $rs<1$.
\end{proof}
From now on, we will always be working with Born-Infeld solitons and BI soliton surfaces defined over such domains in $rs$ plane and consider the points $(r,s)$ with $|rs|<1$ only. Thus, from Note \ref{BI graph}, we get that if for any such $r_0$ and $s_0$, $F'(r_0) G'(s_0)\neq 0$, the equations $\eqref{x}$, $\eqref{y}$ and $\eqref{p}$ give a BI soliton in a neighborhood of $(r_0,s_0)$.

\subsection{Generalized Born-Infeld Soliton Surfaces}

We define generalized Born-Infeld soliton surfaces as surfaces of the form $X(r,s)$  $= (x(r,s),  y(r,s), z(r,s))$, where the components are given by the representation formulae of Barbashov and Chernikov:
\begin{align}
&x(r,s)-y(r,s)=F(r)-\int s^2 G'(s) ds\label{xg}
\\
&x(r,s)+y(r,s)=G(s)-\int r^2 F'(r) dr \label{yg}
\\
&z(r,s)=\int rF'(r) dr + \int sG'(s) ds \label{pg}
\end{align}
for arbitrary $C^2$ functions $F$ and $G$.\\

Note that any BI soliton given by \eqref{x}, \eqref{y} and \eqref{p} in the domain $|rs|<1$ gives a BI soliton surface (which is also a generalized BI soliton surface).

In the next section, we will prove that if a generalized BI soliton surface is a graph over $xy$ plane, then away from the points $\{(r,s): F'(r)G'(s)=0\}$, $1 + \phi_x^2 - \phi_y^2 <0$ and it is a Born-Infeld soliton surface.

\subsection{Conformal Parameters for Generalized\\ BI Soliton Surfaces}
Suppose $X(r,s)=(x(r,s),y(r,s),z(r,s))$ defines a generalized  BI soliton surface on a domain $\Omega$. 

The following theorem gives the conformal parameters for such surfaces explicitly.

\begin{thm}\label{conformal}
For a generalized BI soliton surface $X= (x(r,s),  y(r,s), z(r,s))$, the conformal parameters are 
$\alpha=\frac{r+s}{2}$ and $\beta=\frac{r-s}{2}$.
\end{thm}
\begin{proof}
 Note that $(x+y)_r=-r^2 F'(r)$, $(x+y)_s=G'(s)$, $(x-y)_r=F'(r)$, $(x-y)_s=-s^2 G'(s)$, $z_r=rF'(r)$ and $z_s=sG'(s)$. Hence, $\langle X_r, X_r\rangle_{\B^3}=x_{r}^2-y_{r}^2+z_{r}^2=(x+y)_{r} (x-y)_{r}+z_{r}^2=0$ and similarly, $\langle X_s, X_s\rangle_{\B^3}=0$.\\
As $r=\alpha+\beta$ and $s=\alpha-\beta$, using Chain rule we deduce that $\frac{\partial}{\partial \alpha}=\left(\frac{\partial}{\partial r}+\frac{\partial}{\partial s}\right)$ and $\frac{\partial}{\partial \beta}=\left(\frac{\partial}{\partial r}-\frac{\partial}{\partial s}\right)$. So, we have $(x+y)_{\alpha}$ $=$ $\left(-r^2 F'(r)+G'(s)\right)$, $(x+y)_{\beta}$ $=$ $\left(-r^2 F'(r)-G'(s)\right)$, $(x-y)_{\alpha}=\left(F'(r)-s^2 G'(s)\right)$, $(x-y)_{\beta}=\left(F'(r)+s^2 G'(s)\right)$, $z_{\alpha}=\left(rF'(r)+sG'(s)\right)$ and $z_{\beta}=\left(rF'(r)-sG'(s)\right)$.\\
Thus, we get $\langle X_{\alpha}, X_{\alpha} \rangle_{\B^3} = x_{\alpha}^2-y_{\alpha}^2+z_{\alpha}^2=(x+y)_{\alpha} (x-y)_{\alpha}+z_{\alpha}^2=(rs+1)^2 F'(r)G'(s)$ and similarly, $\langle X_{\beta}, X_{\beta} \rangle _{\B^3}=-(rs+1)^2 F'(r)G'(s)=-\langle X_{\alpha}, X_{\alpha} \rangle_{\B^3}$. Finally, as for any function $f$, $f_{\alpha} f_{\beta}=\left(f_r^2-f_s^2\right)$, we get $\langle X_{\alpha}, X_{\beta}\rangle_{\B^3}=x_{\alpha} x_{\beta}-y_{\alpha} y_{\beta}+z_{\alpha} z_{\beta}=\left(\langle X_r, X_r\rangle_{\B^3}-\langle X_s, X_s\rangle_{\B^3} \right)=0$. So, $\alpha$, $\beta$ are the conformal parameters for the generalized BI soliton surface $X$.
\end{proof}

The following corollary follows from the above proof.
\begin{cor}\label{non-regularity}
For $C^2$ functions $F$ and $G$, if $F'(r)=0$ or $G'(s)=0$ for some $r$ and $s$, then they will correspond to a non-regularity on the surface $X$.
\end{cor}
\begin{proof}
Suppose $(rs+1)^2 F'(r) G'(s)=0$. If $X_{\alpha}$ and $X_{\beta}$ are linearly independent, then they will span the tangent plane. Now, from the proof of Theorem \ref{conformal}, $\langle X_{\alpha}, X_{\alpha}\rangle_{\B^3}=\langle X_{\alpha}, X_{\beta}\rangle_{\B^3}=\langle X_{\beta}, X_{\beta}\rangle_{\B^3}=0$ at that point. Hence, $\langle V, V\rangle_{\B^3} =0$ for any vector $V$ in that tangent space (as any vector in the tangent plane will be a linear combination of $X_{\alpha}$ and $X_{\beta}$). But the set of vectors $V$ in $\B^3$, that satisfy $\langle V,V\rangle=0$ does not contain any plane. So, $X_{\alpha}$ and $X_{\beta}$ are linearly dependent and hence the corresponding point will not be regular.
\end{proof}

The following remark says that if $X=(x(r,s),y(r,s),z(r,s))$ is a generalized BI soliton surface, then $Y=(z(r,s),x(r,s),y(r,s))$ gives a timelike minimal surface away from a certain subset of the domain.

\begin{rem}\label{TMS from BI}
For any $C^2$ functions $F$ and $G$, if we define $x$, $y$ and $z$ by $\eqref{xg}$, $\eqref{yg}$ and $\eqref{pg}$ (here we do not assume that $z$ can be written as a function of $x$ and $y$), then we can define a map $Y:=(z,x,y)$ from the $rs$ plane (or some domain in $rs$ plane) to $\Lo^3$. Then $\langle Y_{\alpha}, Y_{\alpha} \rangle_{\Lo^3} = x_{\alpha}^2-y_{\alpha}^2+z_{\alpha}^2=(rs+1)^2 F'(r)G'(s)=-\langle Y_{\beta}, Y_{\beta}\rangle_{\Lo^3}$ and $\langle Y_{\alpha}, Y_{\beta}\rangle_{\Lo^3}=x_{\alpha} x_{\beta}-y_{\alpha} y_{\beta}+z_{\alpha} z_{\beta}=0$ from the proof of Theorem \ref{conformal}.) As $\frac{\partial}{\partial \alpha}=\left(\frac{\partial}{\partial r}+\frac{\partial}{\partial s}\right)$ and $\frac{\partial}{\partial \beta}=\left(\frac{\partial}{\partial r}-\frac{\partial}{\partial s}\right)$, it is easy to check that for any $C^2$ function $f$, $f_{\alpha\alpha}-f_{\beta\beta}=4 f_{rs}$. Now, as each of $x$, $y$ and $z$ is the sum of a function of $r$ and a function of $s$, $x_{rs}=y_{rs}=\phi_{rs}=0$ and hence, $Y_{\alpha\alpha}-Y_{\beta\beta}=4 Y_{rs}=0$.

So, away from the points $(r,s)$, where $(rs+1)^2 F'(r)G'(s)=0$, $Y$ defines a timelike minimal surface (this follows from the characterization of timelike minimal surface in terms of conformal coordinates).
\end{rem}

For a timelike minimal surface $Y$ given by conformal parameters $\alpha$, $\beta$, we know that  $Y_{\alpha\alpha}-Y_{\beta\beta}=0$. From Remark \ref{BI and TMS}, we know that we can get a timelike minimal surface from a BI soliton surface just by renaming the coordinates. So, it follows naturally that if $X=(x,y,\phi)$ is a BI soliton surface given by \eqref{x}, \eqref{y} and \eqref{p}, then $X_{\alpha\alpha}-X_{\beta\beta}=0$, where $\alpha=\frac{r+s}{2}$ and $\beta=\frac{r-s}{2}$. It can be verified directly as well. Thus we have 

\begin{prop}\label{laplacian}
For a BI soliton $\phi$ given by the general solution $\eqref{x}$, $\eqref{y}$ and $\eqref{p}$, if $\alpha:=\frac{r+s}{2}$ and $\beta:=\frac{r-s}{2}$, then the corresponding BI soliton surface $X$ satisfy $X_{\alpha\alpha}-X_{\beta\beta}=0$.
\end{prop}

\begin{prop}\label{general BI to BI}
If a generalized BI soliton surface $X=(x,y,z)$ is a graph over $xy$ plane, then $z$ as a function of $x,y$ will give a BI soliton on $\{(r,s):|rs|<1\}$ away from the points where $F'(r)G'(s)=0$.
\end{prop}

\begin{proof}
Let $z$ as a function of $x,y$ be given as $\phi=\phi(x,y)$. Now, from Remark \ref{TMS from BI}, $Y=(\phi(x,y),x,y)$ will denote a timelike minimal surface away from points with $(rs+1)^2 F'(r)G'(s)=0$. The normal to $Y$ (in terms of $x$ and $y$) is along $\tilde{n}=Y_x\times_{\Lo^3} Y_y=(1, -\phi_x,\phi_y)$. So, $\phi_x^2-\phi_y^2+1=\langle Y_x,Y_y\rangle_{\Lo^3}>0$ (as $Y$ is timelike, its normal is spacelike). Thus, $\phi$ satisfies the hyperbolicity condition. From Remark \ref{BI and TMS}, as $Y=(\phi(x,y),x,y)$ defines a timelike minimal surface away from points with $(rs+1)^2 F'(r)G'(s)=0$, $\phi$ satisfies the BI equation away from those points. Hence, $\phi$ is a BI soliton away from these points.
\end{proof}

\subsection{Formula for \textit{N} for a BI Soliton Surface}
In this section, we will derive the formula for the unit normal $N$ to a BI soliton surface in terms of $r$ and $s$, where the corresponding BI soliton is given by $\eqref{x}$, $\eqref{y}$ and $\eqref{p}$. It is quite interesting that $N$ does not depend on $F$ and $G$, a result proved in \cite{dev}. More interestingly, given the unit normal $N$ at any point, we can calculate the corresponding point in $rs$ plane. The results are summarized below.

\begin{thm}\label{normal}
For a BI soliton $\phi=\phi(x,y)$ given by $\eqref{x}$, $\eqref{y}$ and $\eqref{p}$, the unit normal $N$ to the corresponding BI soliton surface at a regular parameter point $(r,s)$ is given by $\widetilde{N}:=\left(\frac{r+s}{1+rs}, \frac{r-s}{1+rs}, \frac{rs-1}{1+rs}\right)$ or $-\widetilde{N}$.
\end{thm}
\begin{proof}
From the proof of Theorem \ref{conformal}, we have the partial derivatives of $x+y$, $x-y$ and $\phi$ wrt $r$ and $s$. So, $x_r=\frac{1}{2}\left((x+y)_r+(x-y)_r\right)=\frac{(1-r^2)F'(r)}{2}$ and similarly, $x_s=\frac{(1-s^2)G'(s)}{2}$, $y_r=-\frac{(r^2+1)F'(r)}{2}$ and $y_s=\frac{(s^2+1)G'(s)}{2}$. If the corresponding BI soliton surface is $X=(x,y,\phi)$, then
\begin{align*}
X_r\times_{\B^3} X_s=&\det\left(
\begin{matrix}
\hat{i} & -\hat{j} & \hat{k}\\
x_r & \phantom{-}y_r & \phi_r\\
x_s & \phantom{-}y_s & \phi_s
\end{matrix}
\right)\\
=&-\frac{F'(r) G'(s)(1+rs)}{2}\left((r+s), (r-s), (rs-1)\right).
\end{align*}
So, $\norm{X_r\times_{\B^3} X_s}=\sqrt{\left|\langle X_r\times_{\B^3} X_s, X_r\times_{\B^3} X_s\rangle\right|}=\frac{\left|F'(r)G'(s)\right| (1+rs)^2}{2}$. Thus, $N:=\frac{X_r\times X_s}{\norm{X_r\times X_s}}=-\sign(F'(r)G'(s))\left(\frac{r+s}{1+rs}, \frac{r-s}{1+rs}, \frac{rs-1}{1+rs}\right)=-\sign(F'(r)G'(s))\widetilde{N}$ away from non-regular points. This proves the theorem.
\end{proof}

\begin{note}\label{N in ab}
Upto a sign, $N$ is $\widetilde{N}$. In terms of conformal parameters $\alpha$ and $\beta$, $\widetilde{N}=\left(\frac{2\alpha}{1+\alpha^2-\beta^2},\frac{2\beta}{1+\alpha^2-\beta^2},\frac{\alpha^2-\beta^2-1}{1+\alpha^2-\beta^2}\right)$. Define a surface $U:={(x,y,z):x^2-y^2+z^2=1}$. It is a component of the unit sphere in $\B^3$ containing the point $\nu=(0,0,1)$ (the north pole). For a point $(\alpha,\beta)$, let $\sigma(\alpha,\beta)$ be the point, at which the line joining $(\alpha,\beta,0)$ and $\nu$ intersects $U$. Then $\widetilde{N}=\sigma(\alpha,\beta)$ and $(\alpha,\beta,0)$ is the stereographic projection of $\widetilde{N}$ in $\B^3$. Recall that analogous result holds for maximal and minimal surfaces.

\end{note}
We saw that for a regular parameter point $(r,s)$, $N=\pm \widetilde{N}$. Also note that as $|rs|<1$, $\frac{rs-1}{1+rs}<0$. So, the third component of $\widetilde{N}$ is always negative. Thus, by observing the third coordinate of $N$, we can determine whether $N=\widetilde{N}$ or $N=-\widetilde{N}$ and then we can find the corresponding parameter point $(r,s)$.

\begin{thm}\label{normal to point}
If the unit normal $N$ to a BI soliton surface at a regular parameter point $(r,s)$ is $(n_1,n_2,n_3)$, then $(r,s)=\left(\frac{n_1+n_2}{1-n_3},\frac{n_1-n_2}{1-n_3}\right)$, if $n_3<0$ and $(r,s)=\left(\frac{-(n_1+n_2)}{1+n_3},\frac{-(n_1-n_2)}{1+n_3}\right)$, if $n_3>0$.
\end{thm}
\begin{proof}
Consider the case $n_3<0$. Then $N=\widetilde{N}=\left(\frac{r+s}{1+rs}, \frac{r-s}{1+rs}, \frac{rs-1}{1+rs}\right)$ from Theorem \ref{normal}. So, $n_3=\frac{rs-1}{1+rs}$, that gives $rs=\frac{1+n_3}{1-n_3}$. Substituting this for the first two components, we get that $r+s=(1+rs) n_1=\frac{2n_1}{1-n_3}$ and $r-s=\frac{2n_2}{1-n_3}$, which give $r=\frac{n_1+n_2}{1-n_3}$ and $s=\frac{n_1-n_2}{1-n_3}$.\\
If $n_3>0$, then $N=-\widetilde{N}$. So, $r$ and $s$ will be given by above formulae with $-n_1$, $-n_2$ and $-n_3$ replacing $n_1$, $n_2$ and $n_3$ respectively.
\end{proof}

\begin{note}\label{alt repr}
As $n_1^2-n_2^2+n_3^2=1$, $(n_1+n_2)(n_1-n_2)=(1+n_3)(1-n_3)$. So, $\frac{n_1+n_2}{1-n_3}=\frac{1+n_3}{n_1-n_2}$ and $\frac{n_1-n_2}{1-n_3}=\frac{1+n_3}{n_1+n_2}$.

\end{note}

\section{Bj\"{o}rling Problem for BI Soliton Surface}
\subsection{Existence of Solution to Bj\"orling Problem}
Given a real analytic strip $\{(c(t),n(t)):t\in I\}$, Bj\"{o}rling Problem is to find a BI soliton surface, that contains the curve $c$ and whose unit normal at $c(t)$ is precisely $n(t)$. Here we will present two approaches for finding whether a Bj\"{o}rling problem has a solution or not.

\subsubsection*{Approach 1}
From Remark \ref{BI and TMS}, we know that $(x,y,\phi(x,y))$ will give a BI soliton surface in $\B^3$ if and only if $(\phi(x,y),x,y)$ gives a timelike minimal surface in $\Lo^3$. Thus, the BI soliton surface given by $(x,y,\phi(x,y))$ will solve the Bj\"{o}rling problem with curve $c=(c_1,c_2,c_3)$ and unit normal $n=(n_1,n_2,n_3)$ if and only if the minimal surface given by $(\phi(x,y),x,y)$ solves the Bj\"{o}rling problem with curve $c=(c_3,c_1,c_2)$ and unit normal $n=(n_3,n_1,n_2)$. This gives us the following theorem.

\begin{thm}\label{approach 1}
The Bj\"{o}rling problem for BI soliton surface with the curve $c=(c_1,c_2,c_3)$ and the unit normal $n=(n_1,n_2,n_3)$ has a solution if and only if there is a timelike minimal surface such that\\
(i) it solves the Bj\"{o}rling problem with curve $c=(c_3,c_1,c_2)$ and unit normal $n=(n_3,n_1,n_2)$ and\\
(ii) it can be written as a graph over $yz$ plane in some neighborhood of the curve $c$.
\end{thm}
The proof is evident from the above discussions.  In \cite{cha}, it was shown that the Bj\"{o}rling problem for timelike minimal surface with $(\widetilde{c},\widetilde{n})$ has a solution if $\widetilde{c}$ is a spacelike or timelike curve. So, if ${c'_1}^2-{c'_2}^2+{c'_3}^2>0$ or ${c'_1}^2-{c'_2}^2+{c'_3}^2<0$, there exists a timelike minimal surface satisfying the first condition in Theorem \ref{approach 1}. For completeness, we mention that result here. 

\begin{thm}[Chaves, Dussan \& Magid]
The Bj\"{o}rling problem for timelike minimal surface with the spacelike or timelike curve $\widetilde{c}$ and unit normal $\widetilde{n}$ is given by
\begin{equation}\label{tm}
 X(z)= \emph{Re} \left(\widetilde{c}(w)+k'\int_{s_0}^{w} \widetilde{n}(\xi)\times \widetilde{c}'(\xi) d\xi \right)
\end{equation}
where $z=t+k's$, $w=z$ if $\widetilde{c}$ is timelike and $w=k'z=s+k't$ if $\widetilde{c}$ is spacelike and $t_0$ is any point on $I$. [$k'$ is the analogue of $i$ in split-complex number system with $k'^2=1$.]
\end{thm}
In \cite{gale}, Gale and Nikaid\^{o} gave some sufficient conditions for global invertibility of a differentiable map. 

For example, if the map $(t,s)\mapsto X(t,s)=(X_1(t,s),X_2(t,s),X_3(t,s))$ defines a surface, it can be written in the form $(\phi(x,y),x,y)$ if the Jacobian of the map $(t,s)\mapsto (X_2(t,s),X_3(t,s))$ is positive (or negative) quasi-definite on a convex domain. However, this condition looks quite restrictive. If the above Jacobian is positive (or negative) quasi-definite at $(t,0)$ for $t\in I$ (a closed interval) (or at the points $(0,s)$ for $s$ in a closed interval $I$), then for a $C^2$ map $X$, it will imply that the Jacobian is positive (resp. negative) quasi-definite on a convex domain containing that closed interval as we will see in proof of the next theorem.\\
There are other sufficient conditions of global invertibility in \cite{gale} and those can also be used analogously. We use some of the results from Theorem 7 of that paper to get Theorem \ref{existence}. First we mention the theorem from $\cite{gale}$.

\begin{thm}[Gale \& Nikaid\^{o}]\label{univalence}
(i) Suppose $\Omega$ is an arbitrary rectangular region (either closed or non-closed) and $F$ has continuous partial derivatives. If no principal minors of the Jacobian vanish, then $F$ is injective.\\
(ii) Suppose that $\Omega$ is an open rectangular region and $F$ has continuous partial derivatives. If the Jacobian does not vanish and no diagonal entries of the Jacobian matrix change signs, then $F$ is injective. 
\end{thm}

\begin{note}\label{invertibility}
By Inverse Function Theorem, if a $C^k$ function $f$ is injective and if its Jacobian does not vanish, then $f$ will be invertible from its range and the inverse will be in $C^k$.
\end{note}

\begin{thm}\label{existence}
The Björling problem for BI soliton surface with the curve $c =
(c_1, c_2, c_3)$ (with ${c'_1}^2-{c'_2}^2+{c'_3}^2\neq 0$) and the unit normal $n = (n_1, n_2, n_3)$ defined on a closed interval $I$ has a solution if the Jacobian matrix $J$ of the map $(t,s)\mapsto (X_2,X_3)$ at the points $(t,0)$ (if   ${c'_1}^2-{c'_2}^2+{c'_3}^2<0$) or $(0,s)$ (if  ${c'_1}^2-{c'_2}^2+{c'_3}^2>0$) satisfies any of the following two criteria:\\
(i) no principal minors of $J$ vanish,\\
(ii) $\det(J)$ does not vanish and the diagonal entries of $J$ does not change sign,\\
where $X(z)= \text{Re} \left(\widetilde{c}(w)+k'\int_{s_0}^{w} \widetilde{n}(\xi)\times \widetilde{c}'(\xi) d\xi \right)$ (as given by \eqref{tm}) is the solution of the Björling problem for timelike minimal surface with $\widetilde{c}=(c_3,c_1,c_2)$ and $\widetilde{n}=(n_3,n_1,n_2)$.

\begin{proof}
Suppose \eqref{tm} defines a timelike minimal surface on a domain $\Omega$ (a domain, where the split-holomorphic extensions of $\widetilde{c}$ and $\widetilde{n}$ exist). If the Jacobian $J$ of the map $(t,s)\mapsto (X_2,X_3)$ satisfies $(i)$ or $(ii)$ at the points on $I$ (identified with $I\times {0}$ or ${0}\times I$ accordingly), then from continuity of $J$, $(i)$ or $(ii)$ (resp.) will hold in a neighborhood of those points as well. For each $u\in I$, consider the largest open ball $B(u,r_u)\subseteq \Omega$ around $u$, on which $(i)$ or $(ii)$ (resp.) holds. (Take all the open balls centered at $u$, for which $(i)$ or $(ii)$ (resp.) holds and take their union, which will again be an open ball.) Now, as $r_u$ will be a continuous function\footnote{If a continuous function $f$ is positive or negative in $B(u,r)$ centered at $u$, then for $\norm{v-u}<\epsilon$, $f$ has the same sign in $B(v,r-\epsilon)\subseteq B(u,r)$ (follows from triangle inequality). So, for $\norm{v-u}<\epsilon$, $r_v\geq r_u-\epsilon$ and similarly, $r_u\geq r_v-\epsilon$. So, $r_u$ is a continuous function of $u$.} of $u$ and as $I$ is compact, there is a minimum value of $r_u$. Call this $R$. Then $(i)$ or $(ii)$ (resp.) will hold in $B(u,R)$ for all $u\in I$. Take $V=\cup_{u\in I} {B(u,R)}$. Consider the rectangle $R=(a-R/2,b+R/2)\times (-R/2,R/2)$, if ${c'_1}^2-{c'_2}^2+{c'_3}^2<0$ or $R=(-R/2,R/2)\times (a-R/2,b+R/2)$, if ${c'_1}^2-{c'_2}^2+{c'_3}^2>0$, where $I=[a,b]$. $R$ is an open rectangle containing $I$ and contained in $V$. So, $(i)$ or $(ii)$ (resp.) holds in $R$. From Theorem 7 of \cite{gale}, the map $(t,s)\mapsto (X_2,X_3)$ defined on $R$ is invertible. Hence, both the conditions of Theorem \ref{approach 1} are satisfied and the result follows.
\end{proof}

\end{thm}
In \cite{dev}, Sreedev M worked on this approach ($i.e.$ using the correspondence between BI solitons and time-like minimal surfaces and some sufficient conditions for global invertibility) to get some sufficient conditions for existence of a solution to Bj\"{o}rling problem for BI soliton surfaces.
We will now consider another approach for finding solution to the Björling problem.

\subsubsection*{Approach 2}
From Theorem \ref{normal to point}, given $\{(c(t),n(t)):t\in I\}$, assuming there exists a generalized BI soliton surface $X$ that solves the Björling problem for $c$ and $n$, we can find the parameter point $(r(t),s(t))$ in $rs$ plane, where the normal to the surface is $n(t)$. Then we can try to find the functions $F$ and $G$ of the general solution given by Barbashov and Chernikov, for which $X(r(t),s(t))=c(t)$. There may not exist any function $F$ and $G$, for which this happens. Even if such $F$ and $G$ exist, it may not be possible to write the surface given by as a graph over $xy$ plane.\\  
We want to find a generalized BI soliton surface that solves the Björling problem for $c$ and $\hat{n}$. Assume $\hat{n}_3\neq 0$ on $I$. If ${\hat{n}_3}<0$, take $n=\hat{n}$, otherwise take $n=-\hat{n}$. Then from Theorem \ref{normal to point}, we get $(r(t),s(t))=\left(\frac{n_1(t)+n_2(t)}{1-n_3(t)},\frac{n_1(t)-n_2(t)}{1-n_3(t)}\right)$. From $\eqref{xg}$, $\eqref{yg}$ and $\eqref{pg}$, substituting $(x,y,z)=c$ for $(r,s)=(r(t),s(t))$ and differentiating wrt $t$ we get --\\
\begin{align}
    &c'_1(t)-c'_2(t)=F'(r(t))r'(t)-s^2(t) G'(s(t))s'(t)\label{cx}\\
    &c'_1(t)+c'_2(t)=G'(s(t))s'(t)-r^2(t) F'(r(t))r'(t)\label{cy}\\
    &c'_3(t)=r(t) F'(r(t))r'(t)+ s(t) G'(s(t))s'(t).\label{cp}
\end{align}
Note that if the above three equations are satisfied for some $F$ and $G$, then they will define a surface given by $\eqref{x}$, $\eqref{y}$ and $\eqref{p}$, that will solve the Björling problem (with some fixed constants of integration). \\
From \eqref{cx} and \eqref{cy}, we get --
\begin{align}
    &(c'_1(t)+c'_2(t))+r^2(t)(c'_1(t)-c'_2(t))=(1-r^2(t)s^2(t))G'(s(t))s'(t)\label{cG}\\
    &s^2(t)(c'_1(t)+c'_2(t))+(c'_1(t)-c'_2(t))=(1-r^2(t)s^2(t)) F'(r(t))r'(t)\label{cF}
\end{align}
With some calculations, it can be shown that as $n$ is orthogonal to $c'$, \eqref{cG} and \eqref{cF} are consistent with \eqref{cp}. Also, note that from \eqref{cG} and \eqref{cF}, we can find $F(r(t))$ and $G(s(t))$ by integrating $F'(r(t))r'(t)$ and $G'(s(t))s'(t)$ wrt $t$ respectively:
\begin{align}
    &F(r(t))= \int \frac{s^2(t)(c'_1(t)+c'_2(t))+(c'_1(t)-c'_2(t))}{(1-r^2(t)s^2(t))} dt\label{F final}\\
    &G(s(t))=\int \frac{(c'_1(t)+c'_2(t))+r^2(t)(c'_1(t)-c'_2(t))}{(1-r^2(t)s^2(t))} dt\label{G final}
\end{align}
As $n_3<0$ on $I$, $1-n_3\neq 0$ on $I$. So, the map $t\mapsto (r(t),s(t))$ is a continuous map. So, $I_1=r(I)$ and $I_2=s(I)$ are intervals. Thus, we know the value of $F$ and $G$ on $I_1$ and $I_2$ respectively and using these values, we can construct the surface given by \eqref{x}, \eqref{y} and \eqref{p}, defined on the set $\Omega=(I_1\times I_2) \cap \{(r,s): |rs|< 1\}$ in $rs$ plane. We may also be able to extend $F$ and $G$ to larger domains, however that may not always be the case.\\
If $I$ is a closed interval and $n_1 n'_2-n_2 n'_1\neq \pm n'_3$ on $I$, then we can extend $F$ and $G$. In this case, $r(I)$ and $s(I)$ are compact, hence closed intervals. As $n_3\neq 0$ on $I$, $(1-r^2(t)s^2(t))=-\frac{4n_3}{(1-n_3)^2}\neq 0$. Moreover, the condition $n_1 n'_2-n_2 n'_1\neq \pm n'_3$ assures that $r'(t),s'(t)\neq 0$. Hence, from \eqref{cG} and \eqref{cF}, we deduce that $F'$ and $G'$ are continuous functions on closed intervals $r(I)$ and $s(I)$. So, from Tietze Extension Theorem, we can continuously extend $F'$ and $G'$ to whole of $\mathbb{R}$. Then we can get globally defined continuously differentiable functions $F$ and $G$.

\begin{thm}
Suppose $\{(c(t),n(t)):t\in I\}$ be the data for the Björling problem.  Then there exists a generalized BI soliton surface, for which the Barbashov-Chernikov functions $F$ and $G$ are given by (\ref{F final}) and (\ref{G final}) for $t\in I$.  
\end{thm} 

In the next subsection we discuss non-uniqueness of the BI soliton.

\subsection{Non-uniqueness of Solution to the Björling\\ Problem for TMS \& BI Soliton Surface}
Let $x$, $y$ and $\phi$ be given by $\eqref{x}$, $\eqref{y}$ and $\eqref{p}$ for some $C^2$ functions $F$ and $G$. From Remark \ref{TMS from BI}, $Y:=(\phi, x, y)$ defines a timelike minimal surface away from non-regular points. Take a domain $\Omega$ in $\alpha \beta$ plane ($\alpha$, $\beta$ are conformal parameters defined previously) such that $\langle Y_{\alpha}, Y_{\alpha} \rangle_{\Lo^3}$ does not change sign. So, $Y:\Omega \to \Lo^3$ is a regular timelike minimal surface. Without loss of generality, $\langle Y_{\alpha}, Y_{\alpha} \rangle_{\Lo^3}>0$ (otherwise, take $-F$ in place of $F$). Let $N$ be the Gauss map of $Y$. Take $c(\alpha)=X(\alpha,0)$ and $n(\alpha)=N(\alpha,0)$ for $\alpha \in I$, where $I=[a,b]$ is a non-trivial closed interval such that $I\times \{0\}$ is contained in $\Omega$. Note that $Y$ solves the Björling Problem for timelike minimal surface with $c$ and $n$. Also, note that $(\alpha,\beta)=(\alpha_0,0)\iff r=s=\alpha_0$. So, $(\alpha,\beta)\in I\times \{0\}$ implies that $a\leq r,s\leq b$. Now, take an open interval $J=(c,d)$ disjoint from a neighborhood $I_0$ of $I$ such that $(I_0\cup J)\times \{0\}$ is compactly contained in $\Omega$. Take a smooth non-negative bump function $f$, compactly supported in $J$ such that $f(\frac{c+d}{2})>0$ and $|f'|\leq\frac{|F'|}{2}$ on $J$. Consider the surface $\widetilde{Y}$ on $\Omega$ given by $\eqref{x}$, $\eqref{y}$ and $\eqref{p}$ with $\widetilde{F}:=F+f$ and $\widetilde{G}:=G$ choosing the constants of integration in those equations in such a way that $Y$ and $\widetilde{Y}$ agrees at $(\alpha,\beta)=(a,0)$. Then $\widetilde{Y}$ is also a regular surface. As $F=\widetilde{F}$ and $G=\widetilde{G}$ in $I_0$ and for a small enough neighborhood of $I\times \{0\}$ in $\Omega$, both $r$ and $s$ will be in $I_0$, $Y$ and $\widetilde{Y}$ agree on some neighborhood of $I\times \{0\}$. So, $\widetilde{Y}$ also solves the Björling Problem for timelike minimal surface with $c$ and $n$. Indeed, $Y\neq \widetilde{Y}$ on $\Omega$ as $\widetilde{F}>F$ in a neighborhood of $\frac{c+d}{2}$ and $(\frac{c+d}{2},0)\in \Omega$.  (We can analogously modify $G$ as well and the same conclusions will follow. )

\begin{note} Note that this is not inconsistent with the result proved in \cite{cha}. In that paper, it was shown that the solution to the Bj\"orling problem is unique on a rhombus $R$ (tilted square) in $\Omega$ with opposite vertices at $(a,0)$ and $(b,0)$. This rhombus will correspond to a square $S$ with vertices $(a,a),(a,b),(b,b)$ and $(b,a)$ in $rs$ plane. Indeed, $F=\widetilde{F}$ and $G=\widetilde{G}$ on $[a,b]$. So, for $(r,s)\in S$, $Y$ and $\widetilde{Y}$ agree and the uniqueness holds in $R$.
\end{note}

Thus, we get the following result.

\begin{thm}\label{non-unique}
If $Y:\Omega_1\to \Lo^3$ and $\widetilde{Y}:\Omega_2\to \Lo^3$ are two timelike minimal surfaces solving Björling Problem with $c$ and $n$, then $Y$ and $\widetilde{Y}$ do not necessarily agree on $\Omega_1 \cap \Omega_2$.
\end{thm}

\begin{note}\label{why not unique}
In case of maximal or minimal surface, uniqueness of the solution to Björling Problem comes from the identity theorem for holomorphic functions. However, identity theorem does not hold for real-valued smooth functions. And \eqref{x}, \eqref{y}, \eqref{p} give a family of timelike minimal surfaces $(\phi,x,y)$ in terms of real-valued functions $F$ and $G$. Thus, it is not surprising that we do not get uniqueness of solution to Björling Problem here.
\end{note}

Similar argument can be used for Björling Problem for BI soliton surface as well. (Start with a regular BI soliton surface, perturb the surface a little bit in such a way that it still remains regular and a graph over $xy$ plane and it still solves the Björling Problem.)

\begin{prop}\label{BI not unique}
If $Y:\Omega_1\to \B^3$ and $\widetilde{Y}:\Omega_2\to \B^3$ are two BI soliton surfaces solving Björling Problem with $c$ and $n$, then $Y$ and $\widetilde{Y}$ do not necessarily agree on $\Omega_1 \cap \Omega_2$.
\end{prop}

\begin{center}
\section*{Acknowledgement}
\end{center}
This work was done during the summer of 2022 under S. N. Bhatt Memorial Excellence Fellowship Program offered by International Centre for Theoretical Sciences (ICTS), Bengaluru under the guidance of Prof. Rukmini Dey. I would like to thank her for suggesting the problem to me, for helpful discussions I had with her and for reviewing the paper thoroughly and giving valuable suggestions. I would also like to thank Sreedev Manikoth for the discussions I had with him. Finally, I want to thank ICTS for giving me opportunity to do this work.\\

\end{document}